\documentclass[
final
]{dmtcs-episciences}

\usepackage[utf8]{inputenc}

\usepackage[round]{natbib}

\usepackage{amsthm, amsmath, amssymb, amsfonts}
\usepackage{algorithm, algpseudocode}
\usepackage{bbm, bm}
\usepackage{graphicx}
\usepackage[dvipsnames]{xcolor}
\usepackage{subcaption}
\usepackage{stmaryrd} 
\usepackage[procnames]{listings} 
\usepackage{footnote} 
\makesavenoteenv{tabular} 
\makesavenoteenv{table} 
\usepackage{tikz}
\usetikzlibrary{automata,positioning}
\tikzset{>=latex}

\newcommand{\D}[1]{\ensuremath{\mathrm{\textbf{D}}}}
\renewcommand{\P}{\ensuremath{\mathcal{P}}}

\newcommand{\F}{\ensuremath{\mathcal{F}}}

\newcommand{\x}{\ensuremath{\mathbf{x}}}

\definecolor{keywords}{RGB}{255,0,90}
\definecolor{comments}{RGB}{0,0,113}
\definecolor{myred}{RGB}{160,0,0}
\definecolor{green}{RGB}{0,150,0}

\newtheorem{proposition}{Proposition}
\newtheorem{theorem}{Theorem}
\newtheorem{lemma}{Lemma}
\theoremstyle{definition}
\newtheorem{definition}{Definition}

\author{Jakub Slia\v{c}an\affiliationmark{1}
  \and Walter Stromquist\affiliationmark{2}}
\title{Improving bounds on packing densities of 4-point permutations}
\affiliation{
  Department of Mathematics and Statistics, The Open University, UK\\
  Department of Mathematics, Bryn Mawr College, USA}
\received{2017-4-26}
\revised{2017-11-1}
\accepted{2017-11-21}

\begin{document}

\publicationdetails{19}{2018}{2}{3}{3286}

\maketitle

\newcommand{\dicycle}{
  \begin{tikzpicture}[baseline=-0.3ex,scale=0.2]
    \tikzstyle{vertex}=[circle,fill=black, minimum size=1pt,inner sep=1pt]
    \node[vertex] (v1) at (0,0){};
    \node[vertex] (v2) at (2,0){};
    \node[vertex] (v3) at (1,1.4){};
    \draw[->](v1)--(v2);
    \draw[->](v2)--(v3);
    \draw[->](v3)--(v1);
  \end{tikzpicture}
}

\newcommand{\twochain}{
  \begin{tikzpicture}[baseline=-0.3ex,scale=0.2]
    \tikzstyle{vertex}=[circle,fill=black, minimum size=1pt,inner sep=1pt]
    \node[vertex] (v1) at (0,0){};
    \node[vertex] (v2) at (2,0){};
    \node[vertex] (v3) at (1,1.4){};
    \draw[->](v1)--(v2);
    \draw[->](v2)--(v3);
  \end{tikzpicture}
}

\newcommand{\orcocherry}{
  \begin{tikzpicture}[baseline=-0.3ex,scale=0.2]
    \tikzstyle{vertex}=[circle,fill=black, minimum size=1pt,inner sep=1pt]
    \node[vertex] (v1) at (0,0){};
    \node[vertex] (v2) at (2,0){};
    \node[vertex] (v3) at (1,1.4){};
    \draw[->](v1)--(v2);
  \end{tikzpicture}
}

\newcommand{\outstar}{
  \begin{tikzpicture}[baseline=-0.3ex,scale=0.2]
    \tikzstyle{vertex}=[circle,fill=black, minimum size=1pt,inner sep=1pt]
    \node[vertex] (v1) at (0,0){};
    \node[vertex] (v2) at (2,0){};
    \node[vertex] (v3) at (1,1.4){};
    \draw[->](v3)--(v1);
    \draw[->](v3)--(v2);
  \end{tikzpicture}
}

\newcommand{\digraphacbd}{
  \begin{tikzpicture}[baseline=-0.3ex,scale=0.2]
    \tikzstyle{vertex}=[circle,fill=black, minimum size=1pt,inner sep=1pt]
    \node[vertex] (v1) at (0,0){};
    \node[vertex] (v2) at (2,0){};
    \node[vertex] (v3) at (1,1.4){};
    \node[vertex] (v4) at (3, 1.4){};
    \draw[->](v1)--(v2);
    \draw[->](v1)--(v3);
    \draw[->](v1)--(v4);
    \draw[->](v2)--(v4);
    \draw[->](v3)--(v4);
  \end{tikzpicture}
}

\newcommand{\gridbadc}{
  \begin{tikzpicture}[baseline=0.5ex,scale=0.15]
  \tikzstyle{vertex}=[circle,draw=black, fill=black, minimum size=2pt,inner sep=1pt]

  \draw[gray, very thin] (0,0) grid (9,9);
  \fill[gray] (1,0) rectangle (2,9);
  \fill[gray] (3,0) rectangle (4,9);
  \fill[gray] (5,0) rectangle (6,9);
  \fill[gray] (7,0) rectangle (8,9);

  \fill[gray] (0,1) rectangle (9,2);
  \fill[gray] (0,3) rectangle (9,4);
  \fill[gray] (0,5) rectangle (9,6);
  \fill[gray] (0,7) rectangle (9,8);
  
  \node[vertex] (v1) at (1.5, 3.5){};
  \node[vertex] (v2) at (3.5,1.5){};
  \node[vertex] (v3) at (5.5,7.5){};
  \node[vertex] (v4) at (7.5,5.5){};
  \draw (v1) (v2) (v3) (v4);
  \end{tikzpicture}}

\newcommand{\gridabdc}{
  \begin{tikzpicture}[baseline=0.5ex,scale=0.15]
  \tikzstyle{vertex}=[circle,draw=black, fill=black, minimum size=2pt,inner sep=1pt]

  \draw[gray, very thin] (0,0) grid (9,9);
  \fill[gray] (1,0) rectangle (2,9);
  \fill[gray] (3,0) rectangle (4,9);
  \fill[gray] (5,0) rectangle (6,9);
  \fill[gray] (7,0) rectangle (8,9);

  \fill[gray] (0,1) rectangle (9,2);
  \fill[gray] (0,3) rectangle (9,4);
  \fill[gray] (0,5) rectangle (9,6);
  \fill[gray] (0,7) rectangle (9,8);
  
  \node[vertex] (v1) at (1.5, 1.5){};
  \node[vertex] (v2) at (3.5,3.5){};
  \node[vertex] (v3) at (5.5,7.5){};
  \node[vertex] (v4) at (7.5,5.5){};
  \draw (v1) (v2) (v3) (v4);
  \end{tikzpicture}}

\newcommand{\gridbacd}{
  \begin{tikzpicture}[baseline=0.5ex,scale=0.15]
  \tikzstyle{vertex}=[circle,draw=black, fill=black, minimum size=2pt,inner sep=1pt]

  \draw[gray, very thin] (0,0) grid (9,9);
  \fill[gray] (1,0) rectangle (2,9);
  \fill[gray] (3,0) rectangle (4,9);
  \fill[gray] (5,0) rectangle (6,9);
  \fill[gray] (7,0) rectangle (8,9);

  \fill[gray] (0,1) rectangle (9,2);
  \fill[gray] (0,3) rectangle (9,4);
  \fill[gray] (0,5) rectangle (9,6);
  \fill[gray] (0,7) rectangle (9,8);
  
  \node[vertex] (v1) at (1.5, 3.5){};
  \node[vertex] (v2) at (3.5,1.5){};
  \node[vertex] (v3) at (5.5,5.5){};
  \node[vertex] (v4) at (7.5,7.5){};
  \draw (v1) (v2) (v3) (v4);
  \end{tikzpicture}}


\newcommand{\atau}{
  \begin{tikzpicture}[baseline=0.5ex,scale=0.15]
  \tikzstyle{vertex}=[circle,draw=black,fill=white, minimum size=2pt,inner sep=1pt]
  \node[vertex] (v1) at (0.5, 0.5){};
  \draw (v1);
  \end{tikzpicture}}

\renewcommand{\a}{
  \begin{tikzpicture}[baseline=0.5ex,scale=0.15]
  \tikzstyle{vertex}=[circle,fill=black, minimum size=2pt,inner sep=1pt]
  \node[vertex] (v1) at (0.5, 0.5){};
  \draw (v1);
  \end{tikzpicture}}


\newcommand{\abtau}{
  \begin{tikzpicture}[baseline=0.5ex,scale=0.15]
  \tikzstyle{vertex}=[circle,draw=black, fill=black, minimum size=2pt,inner sep=1pt]
  \node[vertex] (v1) at (0.5, 0.5){};
  \tikzstyle{vertex}=[circle,draw=black, fill=white, minimum size=2pt,inner sep=1pt]
  \node[vertex] (v2) at (1.5,1.5){};
  \draw (v1) (v2);
  \end{tikzpicture}}

\newcommand{\tauab}{
  \begin{tikzpicture}[baseline=0.5ex,scale=0.15]
  \tikzstyle{vertex}=[circle,draw=black, fill=white, minimum size=2pt,inner sep=1pt]
  \node[vertex] (v1) at (0.5, 0.5){};
  \tikzstyle{vertex}=[circle,draw=black, fill=black, minimum size=2pt,inner sep=1pt]
  \node[vertex] (v2) at (1.5,1.5){};
  \draw (v1) (v2);
  \end{tikzpicture}}

\newcommand{\tauba}{
  \begin{tikzpicture}[baseline=0.5ex,scale=0.15]
  \tikzstyle{vertex}=[circle, draw=black, fill=white, minimum size=2pt,inner sep=1pt]
  \node[vertex] (v1) at (0.5, 1.5){};
  \tikzstyle{vertex}=[circle, draw=black, fill=black, minimum size=2pt,inner sep=1pt]
  \node[vertex] (v2) at (1.5, 0.5){};
  \draw (v1) (v2);
  \end{tikzpicture}}

\newcommand{\batau}{
  \begin{tikzpicture}[baseline=0.5ex,scale=0.15]
  \tikzstyle{vertex}=[circle, draw=black, fill=black, minimum size=2pt,inner sep=1pt]
  \node[vertex] (v1) at (0.5, 1.5){};
  \tikzstyle{vertex}=[circle, draw=black, fill=white, minimum size=2pt,inner sep=1pt]
  \node[vertex] (v2) at (1.5, 0.5){};
  \draw (v1) (v2);
  \end{tikzpicture}}


\newcommand{\ab}{
  \begin{tikzpicture}[baseline=0.5ex,scale=0.15]
  \tikzstyle{vertex}=[circle,fill=black, minimum size=2pt,inner sep=1pt]
  \node[vertex] (v1) at (0.5, 0.5){};
  \tikzstyle{vertex}=[circle,fill=black, minimum size=2pt,inner sep=1pt]
  \node[vertex] (v2) at (1.5,1.5){};
  \draw (v1) (v2);
  \end{tikzpicture}}

\newcommand{\ba}{
  \begin{tikzpicture}[baseline=0.5ex,scale=0.15]
  \tikzstyle{vertex}=[circle, fill=black, minimum size=2pt,inner sep=1pt]
  \node[vertex] (v1) at (1.5, 1.5){};
  \tikzstyle{vertex}=[circle, fill=black, minimum size=2pt,inner sep=1pt]
  \node[vertex] (v2) at (0.5, 0.5){};
  \draw (v1) (v2);
  \end{tikzpicture}}


\newcommand{\tauabc}{
  \begin{tikzpicture}[baseline=0.5ex,scale=0.15]
  \tikzstyle{vertex}=[circle,draw=black,fill=white, minimum size=2pt,inner sep=1pt]
  \node[vertex] (v1) at (0.5, 0.5){};
  \tikzstyle{vertex}=[circle,fill=black, minimum size=2pt,inner sep=1pt]
  \node[vertex] (v2) at (1.5,1.5){};
  \node[vertex] (v3) at (2.5, 2.5){};
  \draw (v1) (v2) (v3);
  \end{tikzpicture}}

\newcommand{\ataubc}{
  \begin{tikzpicture}[baseline=0.5ex,scale=0.15]
  \tikzstyle{vertex}=[circle,draw=black,fill=white, minimum size=2pt,inner sep=1pt]
  \node[vertex] (v1) at (1.5, 1.5){};
  \tikzstyle{vertex}=[circle,fill=black, minimum size=2pt,inner sep=1pt]
  \node[vertex] (v2) at (0.5,0.5){};
  \node[vertex] (v3) at (2.5, 2.5){};
  \draw (v1) (v2) (v3);
  \end{tikzpicture}}

\newcommand{\acbtau}{
  \begin{tikzpicture}[baseline=0.5ex,scale=0.15]
  \tikzstyle{vertex}=[circle,draw=black,fill=white, minimum size=2pt,inner sep=1pt]
  \node[vertex] (v1) at (0.5, 0.5){};
  \tikzstyle{vertex}=[circle,fill=black, minimum size=2pt,inner sep=1pt]
  \node[vertex] (v2) at (1.5,2.5){};
  \node[vertex] (v3) at (2.5, 1.5){};
  \draw (v1) (v2) (v3);
  \end{tikzpicture}}

\newcommand{\abc}{
  \begin{tikzpicture}[baseline=0.5ex,scale=0.15]
  \tikzstyle{vertex}=[circle,fill=black, minimum size=2pt,inner sep=1pt]
  \node[vertex] (v1) at (0.5, 0.5){};
  \node[vertex] (v2) at (1.5,1.5){};
  \node[vertex] (v3) at (2.5, 2.5){};
  \draw (v1) (v2) (v3);
  \end{tikzpicture}}

\newcommand{\acb}{
  \begin{tikzpicture}[baseline=0.5ex,scale=0.15]
  \tikzstyle{vertex}=[circle,fill=black, minimum size=2pt,inner sep=1pt]
  \node[vertex] (v1) at (0.5, 0.5){};
  \node[vertex] (v2) at (1.5,2.5){};
  \node[vertex] (v3) at (2.5, 1.5){};
  \draw (v1) (v2) (v3);
  \end{tikzpicture}}

\newcommand{\bac}{
  \begin{tikzpicture}[baseline=0.5ex,scale=0.15]
  \tikzstyle{vertex}=[circle,fill=black, minimum size=2pt,inner sep=1pt]
  \node[vertex] (v2) at (0.5,1.5){};
  \node[vertex] (v1) at (1.5, 0.5){};
  \node[vertex] (v3) at (2.5, 2.5){};
  \draw (v1) (v2) (v3);
  \end{tikzpicture}}

\newcommand{\bca}{
  \begin{tikzpicture}[baseline=0.5ex,scale=0.15]
  \tikzstyle{vertex}=[circle,fill=black, minimum size=2pt,inner sep=1pt]
  \node[vertex] (v1) at (0.5, 1.5){};
  \node[vertex] (v2) at (1.5,2.5){};
  \node[vertex] (v3) at (2.5, 0.5){};
  \draw (v1) (v2) (v3);
  \end{tikzpicture}}

\newcommand{\cba}{
  \begin{tikzpicture}[baseline=0.5ex,scale=0.15]
  \tikzstyle{vertex}=[circle,fill=black, minimum size=2pt,inner sep=1pt]
  \node[vertex] (v3) at (0.5, 2.5){};
  \node[vertex] (v2) at (1.5,1.5){};
  \node[vertex] (v1) at (2.5, 0.5){};
  \draw (v1) (v2) (v3);
  \end{tikzpicture}}

\newcommand{\cab}{
  \begin{tikzpicture}[baseline=0.5ex,scale=0.15]
  \tikzstyle{vertex}=[circle,fill=black, minimum size=2pt,inner sep=1pt]
  \node[vertex] (v1) at (0.5, 2.5){};
  \node[vertex] (v2) at (1.5,0.5){};
  \node[vertex] (v3) at (2.5, 1.5){};
  \draw (v1) (v2) (v3);
  \end{tikzpicture}}


\newcommand{\abcd}{
  \begin{tikzpicture}[baseline=0.6ex,scale=0.1]
  \tikzstyle{vertex}=[circle,fill=black, minimum size=2pt,inner sep=1pt]
  \node[vertex] (v1) at (0.5, 0.5){};
  \node[vertex] (v2) at (1.5, 1.5){};
  \node[vertex] (v3) at (2.5, 2.5){};
  \node[vertex] (v4) at (3.5, 3.5){};
  \draw (v1) (v2) (v3) (v4);
  \end{tikzpicture}}

\newcommand{\bdca}{
  \begin{tikzpicture}[baseline=0.6ex,scale=0.1]
  \tikzstyle{vertex}=[circle,fill=black, minimum size=2pt,inner sep=1pt]
  \node[vertex] (v1) at (0.5, 1.5){};
  \node[vertex] (v2) at (1.5, 3.5){};
  \node[vertex] (v3) at (2.5, 2.5){};
  \node[vertex] (v4) at (3.5, 0.5){};
  \draw (v1) (v2) (v3) (v4);
  \end{tikzpicture}}

\newcommand{\acdb}{
  \begin{tikzpicture}[baseline=0.6ex,scale=0.1]
  \tikzstyle{vertex}=[circle,fill=black, minimum size=2pt,inner sep=1pt]
  \node[vertex] (v1) at (0.5, 0.5){};
  \node[vertex] (v2) at (1.5, 2.5){};
  \node[vertex] (v3) at (2.5, 3.5){};
  \node[vertex] (v4) at (3.5, 1.5){};
  \draw (v1) (v2) (v3) (v4);
  \end{tikzpicture}}


\newcommand{\bcaoba}{
  \begin{tikzpicture}[baseline=0.6ex,scale=0.1]
  \tikzstyle{vertex}=[circle,fill=black, minimum size=2pt,inner sep=1pt]
  \node[vertex] (v1) at (0.5, 1.5){};
  \node[vertex] (v2) at (1.5, 2.5){};
  \node[vertex] (v3) at (2.5, 0.5){};
  \node[vertex] (v4) at (3.5, 4.5){};
  \node[vertex] (v5) at (4.5, 3.5){};
  \draw[gray, very thin] (0,0) grid (5,5);
  \draw (v1) (v2) (v3) (v4) (v5);
  \end{tikzpicture}}

\newcommand{\baoaoba}{
  \begin{tikzpicture}[baseline=0.6ex,scale=0.1]
  \tikzstyle{vertex}=[circle,fill=black, minimum size=2pt,inner sep=1pt]
  \node[vertex] (v1) at (0.5, 1.5){};
  \node[vertex] (v2) at (1.5, 0.5){};
  \node[vertex] (v3) at (2.5, 2.5){};
  \node[vertex] (v4) at (3.5, 4.5){};
  \node[vertex] (v5) at (4.5, 3.5){};
  \draw[gray, very thin] (0,0) grid (5,5);
  \draw (v1) (v2) (v3) (v4) (v5);
  \end{tikzpicture}}

\newcommand{\aobamba}{
  \begin{tikzpicture}[baseline=0.6ex,scale=0.1]
  \tikzstyle{vertex}=[circle,fill=black, minimum size=2pt,inner sep=1pt]
  \node[vertex] (v1) at (0.5, 0.5){};
  \node[vertex] (v2) at (1.5, 3.5){};
  \node[vertex] (v3) at (2.5, 4.5){};
  \node[vertex] (v4) at (3.5, 1.5){};
  \node[vertex] (v5) at (4.5, 2.5){};
  \draw[gray, very thin] (0,0) grid (5,5);
  \draw (v1) (v2) (v3) (v4) (v5);
  \end{tikzpicture}}


\newcommand{\bcaocba}{
  \begin{tikzpicture}[baseline=0.6ex,scale=0.1]
  \tikzstyle{vertex}=[circle,fill=black, minimum size=2pt,inner sep=1pt]
  \node[vertex] (v1) at (0.5, 1.5){};
  \node[vertex] (v2) at (1.5, 2.5){};
  \node[vertex] (v3) at (2.5, 0.5){};
  \node[vertex] (v4) at (3.5, 5.5){};
  \node[vertex] (v5) at (4.5, 4.5){};
  \node[vertex] (v6) at (5.5, 3.5){};
  \draw[gray, very thin] (0,0) grid (6,6);
  \draw (v1) (v2) (v3) (v4) (v5) (v6);
  \end{tikzpicture}}

\newcommand{\bcaobca}{
  \begin{tikzpicture}[baseline=0.6ex,scale=0.1]
  \tikzstyle{vertex}=[circle,fill=black, minimum size=2pt,inner sep=1pt]
  \node[vertex] (v1) at (0.5, 1.5){};
  \node[vertex] (v2) at (1.5, 2.5){};
  \node[vertex] (v3) at (2.5, 0.5){};
  \node[vertex] (v4) at (3.5, 4.5){};
  \node[vertex] (v5) at (4.5, 5.5){};
  \node[vertex] (v6) at (5.5, 3.5){};
  \draw[gray, very thin] (0,0) grid (6,6);
  \draw (v1) (v2) (v3) (v4) (v5) (v6);
  \end{tikzpicture}}

\newcommand{\bcaocab}{
  \begin{tikzpicture}[baseline=0.6ex,scale=0.1]
  \tikzstyle{vertex}=[circle,fill=black, minimum size=2pt,inner sep=1pt]
  \node[vertex] (v1) at (0.5, 1.5){};
  \node[vertex] (v2) at (1.5, 2.5){};
  \node[vertex] (v3) at (2.5, 0.5){};
  \node[vertex] (v4) at (3.5, 5.5){};
  \node[vertex] (v5) at (4.5, 3.5){};
  \node[vertex] (v6) at (5.5, 4.5){};
  \draw[gray, very thin] (0,0) grid (6,6);
  \draw (v1) (v2) (v3) (v4) (v5) (v6);
  \end{tikzpicture}}

\newcommand{\baoabmab}{
  \begin{tikzpicture}[baseline=0.6ex,scale=0.1]
  \tikzstyle{vertex}=[circle,fill=black, minimum size=2pt,inner sep=1pt]
  \node[vertex] (v1) at (0.5, 1.5){};
  \node[vertex] (v2) at (1.5, 0.5){};
  \node[vertex] (v3) at (2.5, 4.5){};
  \node[vertex] (v4) at (3.5, 5.5){};
  \node[vertex] (v5) at (4.5, 2.5){};
  \node[vertex] (v6) at (5.5, 3.5){};
  \draw[gray, very thin] (0,0) grid (6,6);
  \draw (v1) (v2) (v3) (v4) (v5) (v6);
  \end{tikzpicture}}

\newcommand{\acbmax}{
  \begin{tikzpicture}[scale=0.5]
    \draw (0,1.7)--(1.3,3);
    \draw (1.3,1)--(2,1.7);
    \draw (2,0.7)--(2.3,1);
    \draw (2.3,0.5)--(2.5,0.7);
    \draw (2.5,0.4)--(2.6,0.5);
    \draw (2.65,0.3)--(2.7,0.35);
  \end{tikzpicture}}

\newcommand{\acdbmax}{
  \begin{tikzpicture}[scale=0.3]

    \draw (3,6)--(5,8);
    \draw (5,4.7)--(6.3,6);
    \draw (6.3,4)--(7,4.7);
    \draw (7,3.7)--(7.3,4);
    \draw (7.3,3.5)--(7.5,3.7);
    \draw (7.5,3.4)--(7.6,3.5);
    \draw (7.65,3.3)--(7.7,3.35);

    \draw (0,1.7)--(1.3,3);
    \draw (1.3,1)--(2,1.7);
    \draw (2,0.7)--(2.3,1);
    \draw (2.3,0.5)--(2.5,0.7);
    \draw (2.5,0.4)--(2.6,0.5);
    \draw (2.65,0.3)--(2.7,0.35);

    \draw (-1.7,-0.7)--(-1,0);
    \draw (-1,-1)--(-0.7,-0.7);
    \draw (-0.7,-1.2)--(-0.5,-1);
    \draw (-0.5,-1.3)--(-0.4,-1.2);
    \draw (-0.35,-1.4)--(-0.3,-1.35);

    \draw (-2.7,-2.3)--(-2.4,-2);
    \draw (-2.4,-2.5)--(-2.2,-2.3);
    \draw (-2.2,-2.6)--(-2.1,-2.5);
    \draw (-2.05,-2.7)--(-2,-2.65);

    \draw[thick] (-3.05,-3.05)--(-3,-3);
    \draw[thick] (-3.25,-3.25)--(-3.2,-3.2);
    \draw[thick] (-3.45,-3.45)--(-3.4,-3.4);

    \draw (14,2)--(16,2);
    \draw (14,1.5)--(16,1.5);

    \draw (25,6)--(27,8)--(29.7,3.35)--(25,6);
    \draw (20,-2.35)--(25,-2.35)--(25,3.35)--(20,3.35)--(20,-2.35);

  \end{tikzpicture}}

\newcommand{\Amax}{
  \begin{tikzpicture}[baseline=1ex, scale=0.15]
    \draw (0,1.7)--(1.3,3);
    \draw (1.3,1)--(2,1.7);
    \draw (2,0.7)--(2.3,1);
    \draw (2.3,0.5)--(2.5,0.7);
    \draw (2.5,0.4)--(2.6,0.5);
    \draw (2.65,0.3)--(2.7,0.35);

    \draw (3,5)--(5,3);
  \end{tikzpicture}}

\newcommand{\AAmax}{
  \begin{tikzpicture}[baseline=1ex, scale=0.15]
    \draw (0,1.7)--(1.3,3);
    \draw (1.3,1)--(2,1.7);
    \draw (2,0.7)--(2.3,1);
    \draw (2.3,0.5)--(2.5,0.7);
    \draw (2.5,0.4)--(2.6,0.5);
    \draw (2.65,0.3)--(2.7,0.35);

    \draw (2.5,5)--(5,2.5);
  \end{tikzpicture}}

\newcommand{\Bmax}{
  \begin{tikzpicture}[baseline=1ex, scale=0.15]
    \draw (3,3)--(5,5);
    \draw (5,1)--(7,3);

    \draw (1,0)--(2,1);
    \draw (2,-1)--(3,0);
    
    \draw (0,-1.5)--(0.5,-1);
    \draw (0.5,-2)--(1,-1.5);

    \draw (-0.4,-2.2)--(-0.2,-2);
    \draw (-0.2,-2.4)--(0,-2.2);

  \end{tikzpicture}}

\newcommand{\Cmax}{
  \begin{tikzpicture}[baseline=1ex, scale=0.15]

    \draw (3,5)--(5,3);
    \draw (2,3)--(3,2);
    \draw (1,2)--(2,1);
    \draw (-1,1)--(1,-1);
  \end{tikzpicture}}

\newcommand{\Dmax}{
  \begin{tikzpicture}[baseline=1ex, scale=0.15]
    \draw (0,1.7)--(1.3,3);
    \draw (1.3,1)--(2,1.7);
    \draw (2,0.7)--(2.3,1);
    \draw (2.3,0.5)--(2.5,0.7);
    \draw (2.5,0.4)--(2.6,0.5);
    \draw (2.65,0.3)--(2.7,0.35);

    \draw (3,4.7)--(4.3,6);
    \draw (4.3,4)--(5,4.7);
    \draw (5,3.7)--(5.3,4);
    \draw (5.3,3.5)--(5.5,3.7);
    \draw (5.5,3.4)--(5.6,3.5);
    \draw (5.65,3.3)--(5.7,3.35);
  \end{tikzpicture}}

\newcommand{\Dmaxr}{
  \begin{tikzpicture}[baseline=1ex, scale=0.15]
    \draw (0,1.7)--(1.3,3);
    \draw (1.3,1)--(2,1.7);
    \draw (2,0.7)--(2.3,1);
    \draw (2.3,0.5)--(2.5,0.7);
    \draw (2.5,0.4)--(2.6,0.5);
    \draw (2.65,0.3)--(2.7,0.35);

    \draw (4.5,3)--(5.8,4.3);
    \draw (3.8,4.3)--(4.5,5);
    \draw (3.5,5)--(3.8,5.3);
    \draw (3.3,5.3)--(3.5,5.5);
    \draw (3.2, 5.5)--(3.3, 5.6);
    \draw (3.15, 5.6)--(3.2, 5.65);
    

  \end{tikzpicture}}

\newcommand{\Emax}{
  \begin{tikzpicture}[baseline=1ex, scale=0.15]

    \draw (0,2)--(2,0);
    \draw (2,4)--(4,6);
    \draw (4,2)--(6,4);
  \end{tikzpicture}}

\newcommand{\acdbmaxapprox}{
  \begin{tikzpicture}[baseline=1ex, scale=0.5]
    \draw (0,1.7)--(1.3,3);
    \draw (1.3,1)--(2,1.7);
    \draw (2,0.7)--(2.3,1);
    \draw (2.3,0.5)--(2.5,0.7);
    \draw (2.5,0.4)--(2.6,0.5);
    \draw (2.65,0.3)--(2.7,0.35);
  \end{tikzpicture}}

\begin{abstract}
We consolidate what is currently known about packing densities of 4-point permutations and in the process improve the lower bounds for the packing densities of 1324 and 1342. We also provide rigorous upper bounds for the packing densities of 1324, 1342, and 2413. All our bounds are within $10^{-4}$ of the true packing densities. Together with the known bounds, this gives us a fairly complete picture of all 4-point packing densities. We also list a number of upper bounds for small permutations of length at least five. Our main tool for the upper bounds is the framework of flag algebras introduced by Razborov in 2007.
\end{abstract}

\section{Introduction}
\label{sec:intro}

In this paper, we study packing densities of small permutations. A \emph{permutation} is an ordered tuple utilizing all integers from $\{1,\ldots,n\}$. We say that $S = S[1]S[2]\cdots S[m] $ is a \emph{sub-permutation} of $P=P[1]P[2]\cdots P[n]$ if there exists an $m$-subset $\{k_1,\ldots,k_m\}$ of $\{1,\ldots,n\}$ such that for all $1 \leq i,j \leq m$, $S[i] < S[j]$ whenever $P[k_i] < P[k_j]$. We denote by $\#(S,P)$ the number of occurrences of $S$ as a sub-permutation of $P$. Let $\P_n$ be the set of all permutations of length $n$. If $\#(S,n) = \max_{P \in \P_n}\#(S,P)$, then the \emph{packing density} of $S$ is defined to be $p(S) = \lim_{n\to\infty} \#(S,n)/\binom{n}{m}.$

\begin{table}[ht]
\centering
\begin{tabular}{|c | c | c | c | c|}
\hline
$\mathbf{S}$ & \textbf{lower bound} & \textbf{ref LB} & \textbf{upper bound} & \textbf{ref UB}\\
\hline\hline
1234 & 1 & trivial & 1 & trivial\\
\hline
1432 & $\beta$ & \cite{price1997packing} & $\beta$ & \cite{price1997packing}\\
\hline
2143 & $3/8$ & trivial & 3/8 & \cite{price1997packing}\\
\hline
1243 & $3/8$ & trivial & 3/8 & \cite{albert2002packing}\\
\hline
1324 & $0.244^*$ & \cite{price1997packing} & $ -^* $ & \cite{price1997packing}\\
\hline
1342 & $\gamma^*$ & \cite{batkeyev} & $0.1988373^*$ & \cite{balogh2015minimum}\\
\hline
2413 & $\approx 0.104724$ & \cite{presutti2010packing} & $0.1047805^*$ & \cite{balogh2015minimum}\\
\hline
\end{tabular}
\caption{\small{Overview of packing densities for 4-point permutations. Values $\beta$ and $\gamma$ are known exactly: $\beta = 6\sqrt[3]{\sqrt{2}-1}-6/\sqrt[3]{\sqrt{2}-1}+4 \sim 0.423570$, $\gamma = (2\sqrt{3}-3)\beta \sim 0.19657960$. We know that the packing density of 1324 is close to 0.244 but there is no non-trivial upper bound. The items with an ($^*$) asterisk will be updated by the current work.}}
\label{tab:overview}
\end{table}

The study of permutation packing densities began with Wilf's 1992 SIAM address. Galvin (unpublished) soon rediscovered the averaging argument of~\cite{katona1964exists}, thus proving that $p(S)$ exists for all permutations $S$. The original argument was in the setting of graph theory. In 1993, Stromquist, and independently Galvin and Kleitman (both unpublished), found the packing density of 132. Up to symmetry, 132 is the only permutation of length 3 with a non-trivial packing density.

For 4-point permutations and their packing densities, it is useful to consult Table~\ref{tab:overview}. First results for 4-point permutations, including $1324$, $1432$, and $2143$, came as part of the investigation of various \emph{layered patterns} by~\cite{price1997packing}. Later,~\cite{albert2002packing} proved a tight upper bound for 1243, and upper bounds of $2/9$ for both 2413 and 1342. The current lower bound for the packing density of 2413 was given by~\cite{presutti2010packing}. The upper bounds of 0.1047805 and 0.1988373 for 2413 and 1342, respectively, are mentioned in passing in~\cite{balogh2015minimum}. They do not discuss them any further.

It is worthwhile to point out that~\cite{balogh2015minimum} used flag algebras to attack the packing density problem for monotone sequences of length 4. To the best of our knowledge, the only other application of flag algebras to permutation packing, although indirect, is by~\cite{falgas2013applications}. They obtained the inducibility (as packing density is refered to in graph theory) of a 2-star directed graph $\outstar$. Their result implies the known upper bound for the packing density of 132. Later,~\cite{huang2014stars} used an argument exploiting equivalence classes of vertices to extend the result to all directed $k$-stars. This argument was known in the permutations setting since~\cite{price1997packing} used it to establish the packing densities $p(1k\ldots 2)$ for all $k$. Similarly, although flag algebra software package Flagmatic, written by~\cite{flagmatic}, has been available since 2013, it has not previously been used to obtain an upper bound on the packing density of 1324.

Therefore, we decided to use the flag algebras method to collect, enhance, and improve results in permutation packing densities. In addition to the mathematical content, we make available a flag algebras package for permutations, \href{http://jsliacan.github.io/permpack/}{Permpack}, written as a \href{http://sagemath.org}{Sage} script. For more information about the software, follow~\cite{sagemath}. It does all our computations and can be used for further research. Permpack uses syntax similar to Flagmatic, but requires no installation. We hope this makes it more user friendly.

The rest of this paper is structured as follows. The aim of Section~\ref{sec:defs} is to introduce notation and concepts, including the part of flag algebras that we need. While~\cite{razborov2007original} presented flag algebras in the general setting of a universal model theory without constants and function symbols, we choose permutations to be the structures on which we base our exposition. Section~\ref{sec:mainresults} presents the main results of this chapter. We use flag algebras to provide upper bounds for the packing densities of 4-point permutations 1324, 1342, and 2413. We learnt belatedly about the existence of the latter two bounds from~\cite{balogh2015minimum}. Regarding lower bounds, we give a new lower bound construction for the packing density of 1342 that meets our upper bound to within $10^{-5}$. In case of 1324, we provide a lower bound that agrees with the upper bound on the first five decimal places. Section~\ref{sec:packingsmall} gives a list of selected upper bounds to illustrate the potential of the flag algebras method in the area of permutation packing. These results are not best possible, but can be obtained effortlessly by using our flag algebras package Permpack.
\section{Definitions and concepts}
\label{sec:defs}

A \emph{pattern} of length $k$, where $k \leq n$, is a $k$-tuple of distinct integers from $[n] :=\{1,\ldots,n\}$. Pattern of length $n$ is called a \emph{permutation}. We write tuples as strings: 1324 stands for $(1,3,2,4)$. Two patterns $P$ and $S$ of length $k$ are \emph{identical}, if $P[i] = S[i]$ for all $i \in [k]$. They are \emph{order-isomorphic} if for all pairs of indices $i,j$, it holds that $P[i] < P[j]$ implies $S[i] < S[j]$. For a set $I = \{i_1,\ldots,i_m\}$ of $m$ indices from $[n]$, the \emph{sub-pattern} $P[I]$ is the $m$-tuple $P[i_1]P[i_2]\cdots P[i_m]$. By overloading the notation slightly, we also use $P[I]$ to refer to the \emph{subpermutation} of length $m$ which is order-isomorphic to the sub-pattern $P[I]$. 

A \emph{decreasing (increasing) permutation} of length $k$ is the $k$-tuple $k\ldots321$ ($123\ldots k$). A permutation $P$ is \emph{layered}, if it is an increasing sequence of decreasing permutations. To be exact, a layered permutation $P$ is a concatenation of smaller permutations $P = P_1P_2\ldots P_\ell$ such that for all $1 \leq i \leq \ell$, $P_i$ is a decreasing sequence of consecutive integers satisfying the following: if $x \in P_i$ and $y \in P_j$ with $i<j$, then $x<y$. For instance, 321465987 can be partitioned as $321|4|65|987$, so it is layered. On the other hand, 2413 is not layered.

Given $S$ and $P$ of lengths $m$ and $n$, respectively, we let $\#(S, P)$ denote the number of times that $S$ occurs as a subpermutation of $P$. The \emph{density} of $S$ in $P$ is $$p(S,P) = \frac{\#(S,P)}{\binom{n}{m}}.$$ If $n < m$, we set $p(S,P) = 0$. Intuitively, $p(S,P)$ is the probability that a random $m$-set of positions from $[n]$ induces a pattern in $P$ that is order-isomorphic to $S$. For example, $p(12, 132) = 2/3$ as both 13 and 12 are order-isomorphic to 12 while 32 is not.

Let $\F$ be a set of \emph{forbidden} permutations. We say that permutation $P$ is \emph{$\F$-free} if $\#(F,P) = 0$ for all $F \in \F$. Such $P$ is also said to \emph{avoid} $\F$ or be \emph{admissible}. We denote by $\P_n$ the set of all \emph{admissible} permutations of length $n$. It will always be clear from context what $\F$ is. If $\F = \emptyset$, then the admissible set $\P_n$ is the set of all permutations of length $n$. Notice that if $P$ is admissible, then so are all its subpermutations. Most of the work in this paper concerns the case when $\F = \emptyset$. However, the setting remains the same whenever $\F$ is non-empty, and we provide a few examples to this effect. 

\begin{figure}[ht]
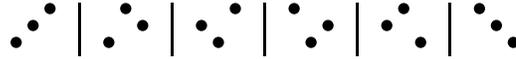

\centering
\scalebox{1.5}{
\begin{tabular}{c|c|c|c|c|c}
\abc & \acb & \bac & \cab &\bca & \cba
\end{tabular}
}
\caption{\small Permutations in $\P_3$, with $\F = \emptyset$. From left to right: 123, 132, 213, 312, 231, 321.}
\label{fig:exP3}
\end{figure}
Let $P \in \P_n$ and $S \in \P_m$ be admissible permutations, and assume $m \leq n$. The maximum value of $p(S,P)$ over $P \in \P_n$ is denoted by $p(S,n)$. Conversely, a permutation $P$ such that $p(S,P) = p(S,n)$ is an $S$-\emph{maximiser} of length $n$. It is well-known that for every $S$, the sequence $\left(p(S,n)\right)_{n\geq 0}$ converges to a value in $[0,1]$ because it is non-increasing and stays between 0 and 1. See~\cite{katona1964exists}. We are now ready to define the quantity that we study, packing density.

\begin{definition}
Let $S$ be a fixed permutation and $\P = \cup_{n\geq 1}\P_n$ the set of admissible permutations. The \emph{packing density} of $S$ is
$$p(S) = \lim_{n\to\infty}p(S,n).$$
\end{definition}
For example, the packing density of 12 in 123-free permutations is $1/2$. Notice that every maximiser of size $n$ has at most two layers. It is then easy to see that they should be of balanced sizes for the packing density to be maximised, i.e.~$\lfloor n/2 \rfloor$ and $\lceil n/2 \rceil$. Let $P_n$ be such balanced 2-layered maximizer of length $n$. Clearly, $p(12, P_n) \to 1/2$ as $n\to\infty$. 

We now formalise the ideas about asymptotic quantities and objects that the discussion is leading to. Let $(P_n)_n = P_1,P_2,P_3,\ldots$ be a sequence of permutations of increasing lengths. We say that $(P_n)_n$ is \emph{convergent} if for every permutation $S$, $(p(S,P_n))_{n=1}^\infty$ converges. A \emph{permuton} $\mu$ is a probability measure with uniform marginals on the Borel $\sigma$-algebra $\mathcal{B}([0,1]^2)$, i.e.~for every $a,b \in [0,1]$ with $a<b$, it holds that $\mu([a,b] \times [0,1]) = b-a = \mu([0,1] \times [a,b])$. See examples of permutons in Figure~\ref{fig:permutons}. 

\begin{figure}[ht]
  \centering
  \begin{subfigure}[b]{0.3\textwidth}
    \centering
      \begin{tikzpicture}[scale=1.8]
      \draw (0,0)--(1,0)--(1,1)--(0,1)--(0,0);
      \draw[thick] (0,0)--(1,1);
      \end{tikzpicture}
      \caption{\small Increasing}
      \label{fig:increasing}
    \end{subfigure}
    \begin{subfigure}[b]{0.3\textwidth}
      \centering
      \begin{tikzpicture}[scale=1.8]
        \fill[draw=black,color=lightgray] (0,0) rectangle (1,1);
        \draw (0,0) rectangle (1,1);
      \end{tikzpicture}
      \caption{\small Lebesgue}
      \label{fig:lebesgue}
    \end{subfigure}
    \begin{subfigure}[b]{0.3\textwidth}
      \centering
      \begin{tikzpicture}[scale=1.8]
        \draw (0,0)--(1,0)--(1,1)--(0,1)--(0,0);
        \draw[thick] (0,0)--(0.5,0.5);
        \draw[thick] (0.5,1)--(1,0.5);
      \end{tikzpicture}
      \caption{\small 1243-maximiser}
      \label{fig:max1243}
    \end{subfigure}
    \caption{\small Examples of permutons. In (a) we have the limit of $(1\ldots n)_{n=1}^\infty$, in (b) it is, with probability one, the limit of a sequence of randomly chosen permutations of each length, and in (c) we have the limit of $(1\ldots \lfloor n/2 \rfloor n \ldots \lceil n/2 \rceil)_{n=1}^\infty$.}
    \label{fig:permutons}
\end{figure}
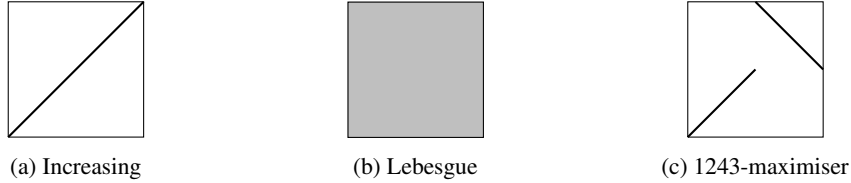

Let $\mu$ be a permuton and $S$ a permutation on $[m]$. One can sample $m$ points from $[0,1]^2$ according to $\mu$ and with probability one they will be in general position (no two aligned vertically or horizontally). We define $p(S,\mu)$ as the probability that a randomly sampled $m$ points from $[0,1]^2$ according to $\mu$ are order-isomorphic to $S$. It turns out that every convergent sequence of permutations has its permuton and vice versa. In particular,~\cite{hoppen2013permlimits} proved that for every $(P_n)_{n\geq 0}$ there exists a unique permuton $\mu$ such that for every $S$, $p(S,\mu) = \lim_{n\to\infty}p(S,P_n)$. In this sense, $\mu$ is the limit of the sequence $(P_n)_n$. In the other direction, they proved that if $\mu$ is a permuton and $P_n$ is a permutation of length $n$ sampled at random according to $\mu$ from $[0,1]^2$, then with probability one the sequence $(P_n)_n$ is convergent (with $\mu$ as its limit). The concept of permutation limits was known as ``packing measures'' since~\cite{presutti2010packing} used them for constructing the 2413 lower bound. In the current work, we use permutons mainly to describe extremal constructions that yield our lower bounds.

\subsection{Flag Algebras}
\label{sec:FA}

The term \emph{flag algebras} refers to a framework first introduced by~\cite{razborov2007original}. It proved to be a very useful tool for researchers in extremal graph theory, but found use in other fields as well. For an overview of results aided by flag algebras, see Razborov's own survey~\cite{razborov2013interim}. For more extensive expositions, see the PhD theses of~\cite{sperfeld2012thesis} and~\cite{volec2014thesis}. By now, there are also many papers with explanations and examples such as~\cite{babertalbot2011jump}, \cite{pikhurko2015neighbourhoods}, \cite{falgasvaughan2012densities}, \cite{falgas2013applications}. For a long list of important results across disciplines of discrete mathematics that were aided by flag algebras, see the abovementioned theses, especially Chapter 1 of~\cite{volec2014thesis}. The main flag algebra result in permutations is~\cite{balogh2015minimum}. In their work on quasirnadom permutations,~\cite{kral2013quasirandom} mention flag algebras as another way to think about the subject. It is important to note that the method of flag algebras has evolved from other combinatorial and analytic methods in combinatorics which had been used by researchers for a long time. The Cauchy-Schwarz type arguments can be found in e.g.~work by~\cite{bondy1997cs} as early as 1990s. The ideas pertinent to quasirandomness have been around since~\cite{chung1988quasirandom}. And while there are other analytic methods that were used successfully to attack extremal problems in combinatorics, the method of flag algebras is syntactical and lends itself to automation. The syntax-based nature of flag algebras is the main feature that distinguishes the theory of flag algebras from the theory of dense graph limits (see e.g.~\cite{lovasz2012networks}). The crux of the method is a systematic conversion of the combinatorial problem into a semidefinite programming problem. The latter can be solved (efficiently) by current SDP solvers. The numerical values returned by the SDP solvers then need to be transformed to exact values (rational or algebraic) to provide valid upper bounds on packing densities.

Before we delve into the method itself, let us consider an example from before. For the remainder of this section, assume that all objects (permutations, flags, types) are admissible unless stated otherwise. Now, assume that we are looking for $123$-free permutations $P$ that are as $12$-dense as possible. We get the following bound without much effort.

\begin{equation}
\begin{aligned}
p(12, P) &= \underbrace{p(12, 123)p(123,P)}_{ = 0} +\ p(12, 132)p(132,P) + p(12, 213)p(213,P)\\
  &+ p(12, 231)p(231,P) + p(12, 312)p(312,P) + p(12,321)p(321,P)\\
  &\leq \max \left\{\frac{2}{3}, \frac{2}{3}, \frac{1}{3}, \frac{1}{3}, 0\right\} = \frac{2}{3}
\label{ex:mantel}
\end{aligned}
\end{equation}
This is strictly better than a trivial bound of 1. However, observe that there is no $P$ of length greater than 4 such that $p(132,P) + p(213,P) = 1$. This follows from Erd\H{o}s-Szekeres theorem (adapted) which states that a permutation of length $(r-1)(s-1)+1$ contains either an increasing subpermutation of length $r$ or a decreasing subpermutation of length $s$. Hence, a permutation of length 5 contains 123 or 321 as a subpermutation. So there are always subsets of size 3 in $P$ which do not induce 132 or 213. Therefore, the bound of 2/3 is unachievable in practice. Knowing this, it would be useful to be able to control how copies of small permutations, such as 132 and 213, interact inside larger permutations. The method of flag algebras helps us systematically take into account the ways in which small patterns overlap inside larger structures. This takes the form of extra coefficients in front of $p(12,P)$ terms in~\eqref{ex:mantel}. If chosen well, they shift weight away from the large values like $p(12,231)$ and $p(12,312)$ and thereby reduce the maximum over all of them.

In general, the process is analogous to the example above. If $S$ is a small permutation whose packing density we seek to determine, we pick a reasonably small value $N \geq |S|$. The crude bound then looks as follows.
\begin{align}
  p(S,P) &= \sum_{P' \in \P_N}p(S,P')p(P',P)\notag\\
  &\leq \max_{P' \in \P_N} p(S,P') \label{eq:crudebd}
\end{align}

Before we describe how exactly we leverage overlaps between small patterns, we need to define flags, types, and operations on them.

\begin{definition}[Flag]
\label{def:permflag}
A permutation $\tau$-\emph{flag} $S^{\tau}$ is a permutation $S$ together with a distinguished subpermutation $\tau$, also called an intersection \emph{type}.
\end{definition}

\begin{figure}[ht]
\centering 
\begin{tabular}{c|c|c|c}
\tauab & \batau & \tauba & \abtau  \\
$S_1^1$ & $S^1_2$ & $S^1_3$ & $S^1_4$
\end{tabular}
\caption{\small If $\tau = 1$ (as permutation), then there are four distinct $\tau$-flags of length two. The empty circle marks $\tau$ in each flag.}
\label{fig:flags1}
\end{figure}

See Figure~\ref{fig:flags1} for a list of all 1-flags on two vertices. The set of all admissible $\tau$ flags of length $m$ is denoted by $\P_m^\tau$. If $\tau$ is the permutation of length 0 or 1, we write $\P_m^0$ and $\P_m^1$, respectively. Notice that $\P_m^0 = \P_m$. The \emph{support} $T$ of $\tau$ in $S^{\tau}$ is the set of indices of $S$ that span $\tau$ in $S^{\tau}$. We say that two permutation flags $S_1^{\tau_1}$ and $S_2^{\tau_2}$ are \emph{type}-isomorphic if $S_1 = S_2$ and if the supports of $\tau_1$ and $\tau_2$ are identical. For instance, in Figure~\ref{fig:flags1}, $S_1^1$ and $S_4^1$ are not type-isomorphic, because the support of $\tau$ in $S_1^1$ is $1$ and in $S^1_4$ it is $2$. For convenience, we set $t :=|\tau|$.

\begin{definition}
Let $S^{\tau}$ be a $\tau$-flag of length $m$, $P^{\tau}$ a $\tau$-flag of length $n \geq m$. We define $\#(S^{\tau}, P^{\tau})$ to be the number of $m$-sets $M \subseteq [n]$ such that $P[M]$ is type-isomorphic to $S^{\tau}$. Flag density is then defined as follows
$$ p(S^{\tau},P^{\tau}) = \frac{\#(S^{\tau}, P^{\tau})}{\binom{n-t}{m-t}}.$$
\end{definition}
In other words, $p(S^{\tau}, P^{\tau})$ is the probability that a uniformly at random chosen subpermutation of length $m$ from $P^{\tau}$, subject to it containing $\tau$, induces a a flag type-isomorphic to $S^{\tau}$. For instance, consider the following flag densities. The empty circle denotes $\tau = 1$. 

\begin{center}
\begin{tabular}{c c c}
$p(\tauab, \tauabc) = 1,$ & $p(\abtau, \tauabc) = 0,$ & $p(\abtau, \ataubc) = 1/2$
\end{tabular}
\end{center}

Finally, we define joint density of two flags, $p(S_1^{\tau}, S_2^{\tau}; P^{\tau})$, as the probability that choosing an $m_1$-set $M_1 \subseteq [n]$ such that $P[M_1]$ contains $\tau$ and choosing an $m_2$-set $M_2 \subseteq [n]$ such that $P[M_2]$ contains $\tau$ and $M_1 \cap M_2 = \tau$ induces $\tau$-flags $P[M_1]^{\tau}$ and $P[M_2]^{\tau}$ in $P^\tau$ which are type-isomorphic to $S_1^{\tau}$ and $S_2^{\tau}$, respectively. The following proposition turns out to be useful (Lemma 2.3 in~\cite{razborov2007original}). It says that choosing subflags with or without replacement makes no difference asymptotically.

\begin{proposition}
Let $S_1^{\tau}$ and $S_2^{\tau}$ be flags on $m_1$ and $m_2$ vertices. Let $n \geq m_1 + m_2 - t$ and $P^{\tau}$ be a flag on $n$ vertices. Then
$$p(S_1^{\tau},P^{\tau})p(S_2^{\tau},P^{\tau}) = p(S_1^{\tau},S_2^{\tau}; P^{\tau}) + o(1),$$
where $o(1) \to 0$ as $n \to \infty$.
\label{thm:littleo}
\end{proposition}

Let $\ell = |\P_m^\tau|$ and fix an order on elements of $\P_m^\tau$. Let $S^\tau_i,S^\tau_j$ be $\tau$-flags from $\P_m^\tau$ and $P$ a $\tau$-flag from $\P_n^\tau$. Furthermore, let $\x$ be a vector with $i$-th entry $p(S_i^{\tau}, P^{\tau})$, and let $Q^\tau$ be a positive semi-definite matrix with dimensions $\ell \times \ell$. Then by Proposition~\ref{thm:littleo} and since $Q^\tau \succeq 0$, we have
$$ 0 \leq \x Q^\tau \x^T = \sum_{i,j \leq \ell} Q^\tau_{ij}p(S_i^{\tau},S_j^{\tau},P^{\tau}) + o(1).$$
Moreover, if we let $\sigma$ be a uniformly at random chosen type in $P$ of length $t$, the inequality above remains true. Moreover, an ``average'' $\sigma$ preserves the non-negativity as well. 
\begin{align}
0 \leq \mathbb{E}_{\sigma}\left(\x Q^\tau \x^T\right) &= \sum_{i,j \leq \ell}Q^\tau_{ij}\frac{1}{\binom{n}{t}}\sum_{\sigma \in \binom{[n]}{t}} p(S_i^{\tau}, S_j^{\tau}; P^{\sigma}) + o(1). \label{eq:fcoeffs}
\end{align}

Next, we write the above expression in terms of permutations on $N$ vertices. Having all information in terms of the same objects allows us to combine it together.  
\begin{align*}
\mathbb{E}_{\sigma}\left(\x Q^\tau \x^T\right) &= \sum_{i,j \leq \ell} Q^\tau_{ij} \frac{1}{\binom{n}{t}}\sum_{\sigma \in \binom{[n]}{t}} \sum_{P' \in \P_N}p(S_i^{\tau}, S_j^{\tau}; (P',\sigma)) p(P',P) + o(1)\\
&=  \sum_{P' \in \P_N}\underbrace{\left(\sum_{i,j \leq \ell} Q^\tau_{ij}\frac{1}{\binom{n}{t}}\sum_{\sigma \in \binom{[n]}{t}}p(S_i^{\tau}, S_j^{\tau}; (P',{\sigma})) \right)}_{\alpha(P', m,\tau)} p(P',P) + o(1)
\end{align*}

Notice that the last expression is of the form $\sum_{P' \in \P_N} \alpha(P',m,\tau)p(P',P)$. There is one of those for each type $\tau$ and value $m$. Every such choice will require another matrix $Q^\tau$. In practice, we first choose $N$, then take all possible pairs of $t$ and $m$ such that $N = 2m-t$. Thus once $N$ is fixed, the choice of $t$ determines the rest. Therefore, let $\alpha(P') = \sum_{\tau} \alpha(P',m,\tau)$ and recall that the expression that we are trying to minimise, subject to $Q^\tau \succeq 0$ for all $\tau$, comes from~\eqref{eq:crudebd}. By adding inequalities of the form of~\eqref{eq:fcoeffs} to~\eqref{eq:crudebd}, we obtain 
\begin{align}
p(S,P) &= \sum_{P'\in \P_N}p(S,P')p(P',P) \notag\\
  &\leq \sum_{P'\in \P_N}p(S,P')p(P',P) + \sum_{P'\in \P_N}\alpha(P')p(P',P)\quad \notag\\
  &\leq \max_{P'\in \P_N}\{p(S,P') + \alpha(P')\}. \label{eq:sdp}
\end{align}
This problem~\eqref{eq:sdp} has the form of a semidefinite programming problem subject to the condition that $Q^\tau \succeq 0$ for every type $\tau$. There exist numerical solvers, such as CSDP or SDPA, that we can use. However, the solution is in the form of numerical PSD matrices. These need to be converted to exact matrices without floating-point entries in a way that preserves their PSD property and still yields a bound that we are satisfied with. Since none of our bounds is tight, we will take a shortcut in rounding. Let $Q'$ be a numerical matrix returned by the solver. Since it is positive semi-definite, it admits a Cholesky decomposition into a lower and upper triangular matrices: $Q' = L'L'^T$. We compute this decomposition and then round the $L'$ matrices into $L$ matrices in such a way that they do not have negative entries on the diagonals. In certificates, we provide these $L$ matrices instead of $Q$ matrices. This way, one can readily check that $Q = LL^T \succeq 0$ by inspecting the diagonal entries of the $L$ matrices. 
\subsection{Example}
\label{sec:example}

The following example is a done-by-hand flag algebras method on a small problem of determining the packing density of 132. We have a lower bound of $2\sqrt{3}-3 \approx 0.464101615\ldots$ given by the standard construction. Assume we want to obtain an upper bound for the packing density of 132. Let $P$ be a (large) 132-maximiser of length $n$ and let $3 \leq \ell \leq n$. By~\eqref{eq:crudebd} we get
\begin{align*}
p(132) &\leq p(132, P)\\
  &= \sum_{P' \in \P_\ell}p(132,P')p(P',P)\\
       &\leq \max_{P' \in \P_\ell} p(132,P').
\end{align*}
We choose $\ell = 3$ and set $\lambda = 2\sqrt{3}-3$. Now consider 
\begin{align*}
\Delta &= \lambda p(123,P) + (\lambda-1)p(132,P) + \lambda p(213,P) + \frac{5\lambda-3}{6}p(231,P)\\
       &+ \frac{5\lambda-3}{6}p(312,P) + \lambda p(321,P).
\end{align*}
Adding the linear combination $\Delta$ of $P'$ densities to the previous crude upper bound improves it to $\lambda$.
\begin{align*}
p(132,P) &\leq \sum_{P' \in \P_\ell}p(132,P')p(P',P) + \Delta\\
&\leq \max_{P' \in \P_\ell}\{\lambda, \lambda,\lambda, \frac{5\lambda-3}{6}, \frac{5\lambda-3}{6},\lambda\}\\
&= \lambda
\end{align*}
The key property of $\Delta$ is that it is non-negative for all $P$, including all $P' \in \P_3$. Let $\sigma$ be a randomly chosen vertex out of the three available. The matrix $Q$ below is positive semi-definite and $\mathbf{x}_{P'}$ is a vector of flag densities for flags in Figure~\ref{fig:flags1}: $$\mathbf{x}_{P'} = \begin{pmatrix}p(\tauab,(P',\sigma))& p(\batau,(P',\sigma))& p(\tauba,(P',\sigma)) & p(\abtau,(P',\sigma))\end{pmatrix}.$$
\begin{align}
Q &= \begin{pmatrix}0 & 0 & 0 & 0 \\ 0 & \lambda & \lambda & 3(\lambda-1)/2\\0 & \lambda & \lambda & 3(\lambda-1)/2\\ 0 & 3(\lambda-1)/2 & 3(\lambda-1)/2 & 3\lambda  \end{pmatrix}
\label{eq:Q}
\end{align}
Averaging over $\sigma$ gives the expression~\eqref{eq:delta} that makes the non-negativity of $\Delta$ apparent.
\begin{align}
\Delta &= \mathbb{E}_\sigma\left(\sum_{P' \in \P_3}\mathbf{x}_{P'}Q\mathbf{x}_{P'}^T\right) \geq 0 \label{eq:delta}
\end{align}
Therefore, we proved that $p(132) \leq 2\sqrt{3}-3$. 

\subsection{Implementation}
Flagmatic 2.0 was written by Emil R. Vaughan and is currently the only general implementation of Razborov's flag algebra framework which is freely available to use and modify. See~\cite{flagmatic} for more information. The project is hosted at \url{http://github.com/jsliacan/flagmatic}. Unfortunatelly, Flagmatic does not support permutations. For this reason, we wrote Permpack, a lightweight implementation of flag algebras on top of SageMath's Sage 7.4 (see~\cite{sagemath}). It does not have all the functionality of Flagmatic but it is sufficient for basic tasks. For more information, code, and installation instructions, see \url{https://github.com/jsliacan/permpack}. 

Let us consider an example of how Permpack can be used on the above example of 132-packing. It will be clear from Permpack's output where the $Q$ matrix above comes from. In Permpack, one needs to specify the complexity in terms of $N$, the length of the admissible permutations which all computations are expressed in terms of. The \texttt{density\_pattern} argument specifies the permutation whose packing density we want to determine. Once permutations, types, flags, and flag products are computed, we can delegate the rest of the tasks to the solver of our choice (currently supported solvers are \texttt{csdp} and \texttt{sdpa\_dd}). The answer is a numerical upper bound on $p(132)$. It can be rounded automatically to a rational bound by the \texttt{exactify()} method of the \texttt{PermProblem} class. The certificate contains admissible permutations, flags, types, matrices $Q$ (as $L$ matrices in the Cholesky decomposition of $Q$) and the actual bound as a rational number (fraction). These are suficient to verify the bound. Below is the script used to obtain the numerical $Q'$ matrix for the packing density of 132 with Permpack.

\lstset{language=Python, basicstyle=\ttfamily\scriptsize, keywordstyle=\color{keywords}, commentstyle=\color{comments}, stringstyle=\color{myred}, showstringspaces=false, identifierstyle=\color{green}, procnamekeys={def,class}, frame=single, caption={Packing 132 with Permpack.}}
\begin{lstlisting}
p = PermProblem(3, density_pattern="132")
p.solve_sdp()
\end{lstlisting}
\lstset{language=Python, basicstyle=\ttfamily\scriptsize, keywordstyle=\color{black}, commentstyle=\color{black}, stringstyle=\color{black}, showstringspaces=false, identifierstyle=\color{black}, procnamekeys={def,class}, frame=single, caption={Output.}}
\begin{lstlisting}
...
Success: SDP solved
Primal objective value: -4.6410162e-01 
Dual objective value: -4.6410162e-01 
Relative primal infeasibility: 5.90e-14 
Relative dual infeasibility: 1.67e-10 
Real Relative Gap: 3.68e-10 
XZ Relative Gap: 6.14e-10 
\end{lstlisting}

It is not difficult to guess the entries of $Q$ from the numerical matrix below, which is part of the output of the SDP solver. The resulting exact matrix $Q$ is shown in~\eqref{eq:Q}.
\lstset{language=Python, basicstyle=\ttfamily\scriptsize, keywordstyle=\color{black}, commentstyle=\color{black}, stringstyle=\color{black}, showstringspaces=false, identifierstyle=\color{black}, procnamekeys={def,class}, frame=single, caption={Floating point $Q'$ matrix.}}
\begin{figure}[ht]
\begin{lstlisting}
  [ 4.55854035127455e-10  6.806084489120e-12  6.8060845047452e-12 -1.032045390820e-10]
  [ 6.80608448912001e-12  0.4641016162301893  0.464101613919741   -0.8038475767936814]
  [ 6.80608450474521e-12  0.464101613919741   0.4641016162301782  -0.8038475767936717]
  [-1.03204539082084e-10 -0.803847576793681  -0.8038475767936717   1.3923048450288649]
\end{lstlisting}
\end{figure}

\section{Results}
\label{sec:mainresults}

The following theorem will be needed later. There exist further variations of it, e.g.~Proposition~2.1 and Theorem~2.2 in~\cite{albert2002packing}. However, we only need the original version.

\begin{theorem}[\cite{stromquist1993unpublished}]
\label{thm:layered}
Let $S$ be a layered permutation. Then for every $n$, extremal value of $p(S,n)$ is achieved by a layered permutation. Moreover, if $S$ has no layer of size 1, every maximiser of $p(S,n)$ is layered.
\end{theorem}

\noindent The scripts used to obtain results in this section can be found at\\

\url{https://github.com/jsliacan/permpack/tree/master/scripts}.\\

\noindent The certificates in support of the upper bounds in this section can be found at the address below. With each result, we provide the name of the certificate file that witnesses it, e.g. \texttt{cert1324.js} witnesses the upper bound for $p(1324)$.\\

\url{https://github.com/jsliacan/permpack/tree/master/certificates}.

\subsection{Packing 1324}
Layered permutations have been studied in depth by~\cite{price1997packing}. He came up with an approximation algorithm that, at $m$-th iteration assumes that the extremal construction has $m$ layers (see Theorem~\ref{thm:layered}) and optimises over their sizes. The algorithm then proceeds to increase $m$ and halts when increasing $m$ does not improve the estimate. In that case, an optimal construction has been found (up to numerical noise from the optimization, if any). In reality, the procedure is stopped manually when approximation is fine enough or the problem becomes too large. Therefore, for every $m$, the value that Price's algorithm gives is a lower-bound for the packing density in question. \\

It is known that the extremal construction for the packing density of 1324 is layered with infinite number of layers. See, for instance,~\cite{albert2002packing} and~\cite{price1997packing}. The main theorem of this section is the following. 
\begin{theorem}
\label{thm:pack1324}
\begin{align*}
0.244054321 < p(1324) < 0.244054549
\end{align*}
\end{theorem}

\begin{proof}
  Consider the construction $\Gamma$ from Figure~\ref{fig:gamma_constr}, where $\Gamma$ is a permuton. Let $C$ denote the middle layer of $\Gamma$ (the largest layer), $B$ denote the layer above (and $B'$ the layer below) $C$, and $A$ denotes the group of the remaining layers above $B$ (and $A'$ denotes the groupp of layers below $B'$). So $\Gamma = A' \oplus B' \oplus C \oplus B \oplus A$, where $A \oplus B$ means that the layer $A$ is entirely below and to the left of the layer $B$. Let $c = |C|$, $b = |B| = |B'|$, and $a =  |A| = |A'|$. We assume that $A$ (and $A'$) is isomorphic to a maximiser for the packing of 132-pattern (213-pattern). The aim is to optimise over $a$ and $b$. Ideally, the tails of $\Gamma$ would also be optimised over, but that is infeasible. So we assume the tails are 132 (213) maximisers. It turns out that the first two steps give a good lower bound. We now compute the density of 1324 patterns in $\Gamma$. There are four distinct (i.e.~up to symmetry) positions that a copy of 1324 can assume in $\Gamma$. Let $xyzw$ be the four points in $\Gamma$ that form a copy of $1324$ in that order.
\begin{enumerate}
\item $y,z \in C$, $x \in A' \cup B'$, $w \in A \cup B$, there are $N_1$ such copies
\item $y,z \in B$, $x \in A' \cup B' \cup C$, $w \in A$, there are $N_2$ such copies
\item $y,z,w \in A$, $x \in A' \cup B' \cup C \cup B$, there are $N_3$ such copies
\item $x,y,z,w \in A$, , there are $N_4$ such copies
\end{enumerate}
Let us now determine quantities $N_1,\ldots, N_4$. 
\begin{enumerate}
\item $N_1 = c^2/2 + (a+b)^2$
\item $N_2 = b^2/2 + a(a+b+c)$
\item $N_3 = (2\sqrt{3}-3)\frac{a^3}{6}\cdot (a+2b+c)$
\item $N_4 = \sum_{k=0}^\infty \frac{\sqrt{3}\cdot(2\sqrt{3}-3)}{6 \cdot (\sqrt{3}+1)^{4k+4}}\cdot a^4$.
\end{enumerate}
Finally, we get the density of 1324 pattern in $\Gamma$. Let $b = 1-c-2a$. Then
\begin{align*}
p(1324, \Gamma) &= \max_{\substack{0< c\leq 1/2\\ 0 < a < \leq 1/4}}24\cdot(N_1 + 2N_2 + 2N_3 + 2N_4)\\
& > 0.244054321.
\end{align*}
This proves the lower bound in Theorem~\ref{thm:pack1324}, because $0.244054321 < p(1324,\Gamma) \leq p(1324)$.
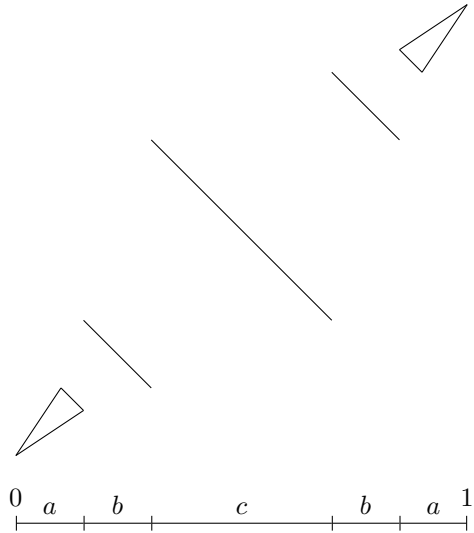
\begin{figure}[ht]
\centering
\begin{tikzpicture}[scale=0.3]
\draw (7,8)--(8,7) (10,10) edge (7,8) (10,10) edge (8,7); 
\draw (4,7)--(7,4);
\draw (-4,4)--(4,-4);
\draw (-4,-7)--(-7,-4);
\draw (-7,-8)--(-8,-7) (-10,-10) edge (-7,-8) (-10,-10) edge (-8,-7); 
\draw[|-] (-10,-13)--(-7,-13) node[above=3pt, at start]{$0$} node[midway, above]{$a$};
\draw[|-] (-7,-13)--(-4,-13) node[midway, above]{$b$};
\draw[|-] (-4,-13)--(4,-13) node[midway, above]{$c$};
\draw[|-] (4,-13)--(7,-13) node[midway, above]{$b$};
\draw[|-|] (7,-13)--(10,-13) node[midway, above]{$a$} node[above=3pt]{$1$};
\end{tikzpicture}
\caption{\small{Permuton $\Gamma$ provides a lower bound for $p(1324)$. The triangles at the ends represent permutons that are maximisers for the packing of 132 and 213 (L to R).}}
\label{fig:gamma_constr}
\end{figure}

We use Flagmatic to prove the upper bound. Since 1324 is layered, there is a 1324-maximiser that is layered as well. Therefore, we can limit the search space to the layered permutations. Since Flagmatic does not work with permutations, we transformed the problem to an equivalent problem in directed graphs -- which Flagmatic can handle. 
\begin{lemma}
\label{lem:permstographs}
Let $\F = \{\dicycle, \twochain, \orcocherry\}$ be the set of forbidden digraphs.  The packing density of 1324 equals the Tur\'an $\digraphacbd$-density of $\F$. In other words, $$p(1324)= p(\digraphacbd, \F).$$
\end{lemma}
\begin{proof}[\ref{lem:permstographs}]
There is a unique way to encode a layered permutation $P$ as a directed graph $D$. If and only if two points $x,y \in P$ form a $12$ pattern, then $xy$ is an arc $x \to y$ in $D$. Forbidding $\dicycle$, $\twochain$, and $\orcocherry$ in $D$ forces it to be a union of independent sets with arcs between them so that if $x,y$ are vertices in one independent set and $u,v$ are vertices in another independent set of $D$, then if $xu$ is an arc in $D$, so are $xv$, $yu$, and $yv$. In other words, all arcs between two independent sets are present, and all go in the same direction. Moreover, the direction is transitive (\dicycle is forbidden). Together with the first rule about the direction of arcs between independent sets, this fully characterizes the digraph $D$ from the permutation $P$. Clearly, the process is reversible.
\end{proof}

Given Lemma~\ref{lem:permstographs}, we use flag algebra method on directed graphs to compute an upper bound for the packing density of \digraphacbd (an equivalent of 1324 in digraphs) over $\{\dicycle, \twochain, \orcocherry\}$-free digraphs. The resulting bound is the one in Theorem~\ref{thm:pack1324}. The certificate is called \texttt{cert1324flagmatic.js}. Note that this is a Flagmatic certificate and can be verified using the \texttt{inspect\_certificate.py} script that comes with Flagmatic. The script is \texttt{pack1324flagmatic.sage}.


\end{proof}

A similar bound can be achieved by Permpack. In particular, we can show that $p(1324)< 0.244054540$. Certificate: \texttt{cert1324permpack.js}. Script: \texttt{pack1324permpack.sage}. Despite Permpack being able to prove a good bound, we used Flagmatic in the proof above to emphasise that this result had been available before Permpack was written.

\subsection{Packing 1342}
\label{sec:pack1342}

The previous lower bound for the packing density of 1342 was approximatelly 0.1965796. The result of~\cite{batkeyev} can be found in~\cite{albert2002packing}.

\begin{figure}[ht]
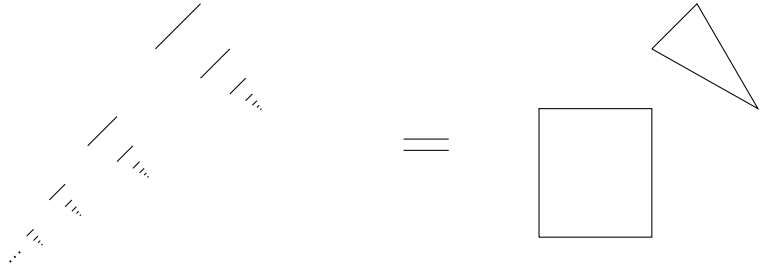

\centering \acdbmax
\caption{\small On the left is Batkeyev's construction for the lower bound on $p(1342)$ as product of packing densities of 132 and 1432. On the right is the schematic drawing of it. The triangle stands for a 231-maximiser and the square stands for the part inside which the entire construction is iterated.}
\label{fig:batkeyev}
\end{figure}

Let $\lambda = 2\sqrt{3}-3$ be the packing density of 231 and $\kappa$ the ratio between the top layer  and the rest of the 1432-maximiser, see~\cite{price1997packing} ($\kappa$ is the root of $3x^4-4x+1$). Batkeyev suggested to replace each layer in the maximiser of 1432 by a 231-maximiser while preserving the size ratio $\kappa$. The density of 1342 in Batkeyev's construction (see Figure~\ref{fig:batkeyev}) is 
\begin{align*}
p(1342, B) &=(8 \sqrt{3}-12)\cdot \sum _{n=0}^{\infty } (1-\kappa)^3 \kappa ^{4n+1} \\
           &= p(132)p(1432)\\
           &= 2 \left(2 \sqrt{3}-3\right) \left(3 \sqrt[3]{\sqrt{2}-1}-\frac{3}{\sqrt[3]{\sqrt{2}-1}}+2\right)\\
           &\approx 0.1965796\ldots
\end{align*}
This lower bound was widely regarded as possibly optimal. Our contribution to this problem is finding a vastly better lower bound construction. However, if we restrict the space of admissible permutations to those that avoid 2431, then Batkeyev's construction is likely optimal. We are able to prove the following theorem on $N=6$ admissible graphs to keep the SDP small (if $N=7$ was chosen, the bound would likely be slightly better).

\begin{theorem}
$$p(1342,\{2431\}) < 0.19658177.$$
\end{theorem}
\begin{proof}
Certificate: \texttt{cert1342\_forb2431.js}. Script: \texttt{pack1342\_forb2431.sage}.
\end{proof}

The following result addresses the actual packing density of 1342 without any forbidden patterns.
\begin{theorem}
\label{thm:newbounds1342}
$$0.198836597 < p(1342) < 0.198837287.$$
\end{theorem}

\begin{proof}
The new lower bound is given by the construction $\Pi$ in Figure~\ref{fig:pack1342}. The weights we used for the parts are given in \texttt{cert1342lb.txt}, located with other certificates.
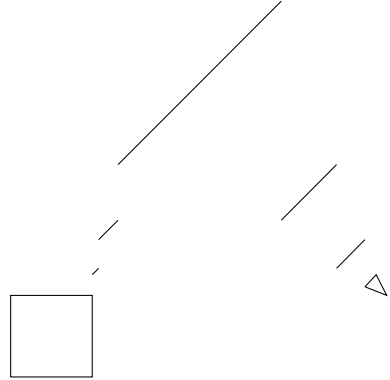
\begin{figure}[ht]
\centering
\begin{tikzpicture}[scale=5]
  
  \draw (0,0)--(0.217,0)--(0.217,0.217)--(0,0.217)--(0,0);
  \draw (0.217,0.272)--(0.234,0.289);
  \draw (0.234,0.3654)--(0.2856,0.417);
  \draw (0.2856, 0.565)--(0.7196,1);
  \draw (0.7196, 0.417)--(0.8666,0.565);
  \draw (0.8666,0.289)--(0.942,0.3654);
  \draw (0.942,0.24)--(0.972,0.272)--(1,0.217)--(0.942,0.24);
\end{tikzpicture}
\caption{\small New lower bound construction $\Pi$ for the packing density of 1342. The part sizes, left to right, are approximatelly 0.2174, 0.0170, 0.0516, 0.4341, 0.1480, 0.0764, 0.0554. The square part represents the part inside which the whole construction is iterated. The triangle part is the extremal construction for 231-packing.}
\label{fig:pack1342}
\end{figure}

\begin{equation}
\begin{aligned}
a_1 &= 0.2174127723536347308692444843\\
a_2 &= 0.0170598057899242722740620549\\
a_3 &= 0.0516101402487892270230230972\\
a_4 &= 0.4340722809873864994312953007\\
a_5 &= 0.1479895625950390496250611829\\
a_6 &= 0.0764457255805656971383351365\\
a_7 &= 0.0554097124446605236389787433
\label{weights}
\end{aligned}
\end{equation}
Label the 7 parts of $\Pi$ from left to right as $a_1,\ldots,a_7$. We assign the weights to them roughly as in~\eqref{weights}. Then a straightforward calculation of the 1342 density in $\Pi$ implies the desired lower bound. The sage script that does this is called \texttt{lb1342.sage}, located with other scripts. The upper bound certificate is called \texttt{cert1342.js}. Script is \texttt{pack1342.sage}.
\end{proof}

 The upper bound obtained without flag algebras stands at $2/9$, see~\cite{albert2002packing}. The upper bound above was obtained via the flag algebras method and confirms the claimed bound from~\cite{balogh2015minimum}. We used $N = 7$ for our computations. While it is possible that $N=8$ would yield a slightly better bound, the computations would be much more expensive. Without a candidate for an exact lower bound, we were satisfied with the bound we obtained with $N=7$.

\subsection{Packing 2413}
\label{sec:pack2413}
The case packing 2413 patterns is fairly complicated as can be seen from the lower bound construction by~\cite{presutti2010packing}. The previous upper bound obtained without flag algebras was $2/9$ by~\cite{albert2002packing}. The bound below was obtained via flag algebras and is in the same range as the bound in~\cite{balogh2015minimum}. 
\begin{theorem}
\label{thm:high2413}
\begin{align*}
p(2413) &< 0.10478046354353523761779.
\end{align*}
\end{theorem}
\begin{proof}
Certificate: \texttt{cert2413.js}. Script: \texttt{pack2413.sage}.
\end{proof}

We used admissible permutations of length $N=7$. Again, larger $N$ could yield a slightly better upper bound, but without an exact lower bound, this effort would not be justified.

\section{Packing other small permutations}
\label{sec:packingsmall}

The flag algebras method will yield upper bounds for many problems. In some cases these bounds are particularly interesting because they are close to their corresponding lower bounds. In this section we list a selection of upper and lower bounds that are potentially sharp since their values appear to be close to each other. 

In the list below we choose to represent the permutations by their drawings in the grid. This is more transparent as the permutations became larger. The extremal constructions (permutons) on the left-hand side of the Table~\ref{tab:otherperms} are represented by their drawings as well. The lower bounds are given on the left-hand side of the table and upper bounds on the right-hand side of the table. This is a sample of the results obtained with Permpack via flag algebras.
\begin{table}[ht]
\centering
\begin{tabular}{l | l | l}
  Permutation & Lower bounds & Upper bounds\\
  \hline
  23154 & $p\left(\bcaoba\ ,\ \Amax\right) = 5!\frac{(2/5)^2}{2!}\frac{(3/5)^3}{3!}(2\sqrt{3}-3)$ & $p\left(\bcaoba\right) \leq 0.16039\ldots$\\
  14523 & $p\left(\aobamba\ ,\ \Bmax\right) \sim 0.153649\ldots$  & $p\left(\aobamba\right) \leq 0.153649\ldots$\\
  21354 & $p\left(\baoaoba\ ,\ \Cmax\right) \sim 0.16515\ldots$  & $p\left(\baoaoba\right) \leq 0.16515\ldots$\\
  231654 & $p\left(\bcaocba\ ,\ \AAmax\right) = 6!\frac{(1/2)^6}{3!^2}(2\sqrt{3}-3)$ & $p\left(\bcaocba\right) \leq 0.145031\ldots$\\
  231564 & $p\left(\bcaobca\ ,\ \Dmax\right)=  (2\sqrt{3}-3)^2\frac{6!}{48^2}$ & $p\left(\bcaobca\right) \leq 0.0673094$\\
  231645 & $p\left(\bcaocab\ ,\ \Dmaxr\right)= (2\sqrt{3}-3)^2\frac{6!}{48^2}$ & $p\left(\bcaocab\right) \leq 0.0673094$\\
  215634 & $p\left(\baoabmab\ ,\ \Emax\right) = \frac{6!}{9^32^3}$ & $p\left(\baoabmab\right) \leq 0.123456\ldots$
\end{tabular}
\caption{\small Exact values are known for all densities on the left-hand side. They are described in the text as they are not easy to write down.}
\label{tab:otherperms}
\end{table}

We now give the descriptions of the lower bound constructions. For $23154 = \bcaoba$, the construction is a sum of two parts in ratio $2:3$ top to bottom. The bottom part is a 231-maximiser while the top part is a simple decreasing segment. Certificate: \texttt{cert23154.js}. Script: \texttt{pack23154.sage}.

The construction for $14523 = \aobamba$ is designed as follows. Let $\alpha$ be the maximiser of $5(1-x)^4/(1-x^5)$ such that $\alpha \in [0,1]$. The topmost sum-indecomposable part of the $\aobamba$-maximiser has length $\alpha$ and the remainder of the maximiser has length $(1-\alpha)$. The construction is iterated inside the part of length $(1-\alpha)$. The part of length $\alpha$ is a skew-sum of two balanced increasing segments. The exact value of the density on the left-hand side of Table~\ref{tab:otherperms} is too complicated to fit in. Certificate: \texttt{cert14523.js}. Script: \texttt{pack14523.sage}.

The construction for $21354 = \baoaoba$ is a 4-layered permuton with layers of lengths $\beta, 1/2-\beta, 1/2-\beta, \beta$, top to bottom. Here, $\beta$ is the real root of $40x^3 - 32x^2 + 9x - 1 = 0$. Again, we only write the approximate value on the left-hand side in Table~\ref{tab:otherperms} for space reasons. Certificate: \texttt{cert21354.js}. Script: \texttt{pack21354.sage}.

The construction for $231654 = \bcaocba$ is identical in structure to the construction for $\bcaoba$, except the ratios of the two parts in the sum are $1:1$. Certificate: \texttt{cert231654.js}. Script: \texttt{pack231654.sage}.

The construction for $231564 = \bcaobca$ is the sum of two 231-maximisers of equal size. In case of $231645 = \bcaocab$, the top 231-maximiser is flipped accordingly. Certificates: \texttt{cert231564.js} and \texttt{cert231645.js}. Scripts: \texttt{pack231564.js} and \texttt{pack231645.js}.

The construction for $215634 = \baoabmab$ has three segments of equal length arranged as portrayed in Table~\ref{tab:otherperms}. Certificate: \texttt{cert215634.js}. Script: \texttt{pack215634.sage}.

\section{Conclusion}

While we now know the packing densities of all 4-point permutations with accuracy of 0.01\%, finding candidates for optimal constructions for the cases of 1324 and 1342 remains a challenge. In the case of 1324, a new idea for the part ratios will be needed to come up with a possible extremal construction. As for the 1342 pattern, the extremal construction might use a different layer formation than our $\Pi$. Even if $\Pi$ has the right structure, the part ratios remain to be determined precisely. The latest status of 4-point packing densities is depicted in Table~\ref{tab:updated}.

\begin{table}[ht]
\centering
\begin{tabular}{|c | c | c | c | c|}
\hline
$\mathbf{S}$ & \textbf{lower bound} & \textbf{ref LB} & \textbf{upper bound} & \textbf{ref UB}\\
\hline\hline
1234 & 1 & trivial & 1 & trivial\\
\hline
1432 & $\beta$ & \cite{price1997packing} & $\beta$ & \cite{price1997packing}\\
\hline
2143 & $3/8$ & trivial & 3/8 & \cite{price1997packing}\\
\hline
1243 & $3/8$ & trivial & 3/8 & \cite{albert2002packing}\\
\hline
1324 & $0.244054321^*$ & -- & $0.244054549^*$ & -- \\
\hline
1342 & $0.198836597^*$ & --  & $0.198837286342^*$ & -- \\
\hline
2413 & $\approx 0.104724$ & \cite{presutti2010packing} & $0.104780463544^*$ & -- \\
\hline
\end{tabular}
\caption{\small{Overview of packing densities for 4-point permutations given the information in this paper. The values with asterisk have been updated.}}
\label{tab:updated}
\end{table}

After 4-point permutations, there are many packing densities of small permutations of length $5, 6,\ldots$. The values of lower bounds and upper bounds in Table~\ref{tab:otherperms} should be made to match. In some cases this will be easier than in others. In particular, the packing density of 21354 has been mentioned in both~\cite{albert2002packing} and~\cite{hasto2002packing}.

There are analogous questions to be asked about packing densities when certain patterns are forbidden. As an example, we mentioned $p(1342,\{2431\})$ in relation to $p(1342)$. 

Next, an interesting line of enquiry was made precise as Conjecture 9 in~\cite{albert2002packing}. For a packing of pattern $S$, is there an extremal construction with infinite number of layers? Are all extremal constructions of that form? More precisely, let an $S$-\emph{maximiser} be an $n$-permutation $P$ such that $p(S,n) = p(S,P)$. If $L_n$ is the number of layers in a layered maximiser of length $n$, what can we say about $L_n$ as $n\to\infty$? For example, we know that the number of layers in every 1324-maximiser is unbounded as $n \to \infty$. We also know that a 2143-maximiser has only two layers, regardless of $n$.

\acknowledgements
We would like to thank Robert Brignall for heaps of useful discussions.

\bibliographystyle{abbrvnat}
\bibliography{packing4point}

\end{document}